\documentclass[12pt]{amsart}
\usepackage{amssymb}
\newtheorem{theorem}{Theorem}
\newtheorem{lemma}[theorem]{Lemma}
\newtheorem{proposition}[theorem]{Proposition}
\newtheorem{corollary}[theorem]{Corollary}

\newtheorem{question}{Question}[section]
\newtheorem{definition}[theorem]{Definition}

\title[A common extension]{A common extension of Arhangel'skii's Theorem and the Hajnal-Juh\'asz inequality}
\author[A. Bella] {Angelo Bella}

\thanks{The research that led to the present paper was partially
supported by a grant of the group GNSAGA of INdAM}

\address{ Dipartimento di Matematica e Informatica, viale A. 
Doria 6, 95125 Catania, Italy}
\email{bella@dmi.unict.it}

\author[S. Spadaro]{Santi Spadaro}

\address{ Dipartimento di Matematica e Informatica, viale A. 
Doria 6, 95125 Catania, Italy}
\email{santidspadaro@gmail.com}

\begin{document}

\begin{abstract} {We present a bound for the weak Lindel\"of number of the $G_\delta$-modification of a Hausdorff space which implies various known cardinal inequalities, including the following two fundamental results in the theory of cardinal invariants in topology:  $|X|\le 2^{L(X)\chi(X)}$
(Arhangel'ski\u\i) and $|X|\le 2^{c(X)\chi (X)}$
(Hajnal-Juh\'asz). This solves a question that goes back to Bell, Ginsburg and Woods \cite{BGW} and is mentioned in Hodel's survey on Arhangel'ski\u\i's Theorem \cite{H}. In contrast to previous attempts we do not need any separation axiom beyond $T_2$.}
\end{abstract}
\maketitle

\section{Introduction}
Two of the milestones in the theory of cardinal invariants in topology are the following inequalities:

\begin{theorem} (Arhangel'ski\u\i, 1969) \cite{A} If
$X$ is
a $T_2$ space, then $|X|\le 2^{L(X)\chi(X)}$. \end{theorem}
\begin{theorem} (Hajnal-Juh\'asz, 1967) \cite{HJ} If
$X$ is
a $T_2$ space, then $|X|\le 2^{c(X)\chi(X)}$. \end{theorem}

Here $\chi(X)$ denotes the \emph{character} of $X$, $c(X)$ denotes the \emph{cellularity} of $X$, that is the supremum of the cardinalities of the pairwise disjoint collection of non-empty open subsets of $X$ and $L(X)$ denotes the \emph{Lindel\"of degree} of $X$, that is the smallest cardinal $\kappa $ such that every open cover of $X$ has a subcover of size at most $\kappa$.

The intrinsic difference between the cellularity and the Lindel\"of degree makes it non-trivial to find a common extension of the two previous inequalities. The first attempt was done in 1978 by Bell, Ginsburg and Woods \cite{BGW}, who used the notion of weak Lindel\"of degree. The weak Lindel\"of degree of $X$ ($wL(X)$) is defined as the least cardinal $\kappa$ such that every open cover of $X$ has a $\leq \kappa $-sized subcollection whose union is dense in $X$. Clearly, $wL(X) \leq L(X)$ and we also have $wL(X) \leq c(X)$, since every open cover without $<\kappa$-sized dense subcollections can be refined to a $\kappa$-sized pairwise disjoint family of non-empty open sets by an easy transfinite induction. Unfortunately, the Bell-Ginsburg-Woods result needs a separation axiom which is much stronger than Hausdorff.

\begin{theorem}\label{3} \cite{BGW} If $X$ is a normal space,
then
$|X|\le 2^{wL(X)\chi(X)}$. \end{theorem}

It is still unknown whether this inequality is true for regular spaces, but in \cite{BGW} it was shown that it may fail for Hausdorff spaces. Indeed, the authors constructed Hausdorff non-regular first-countable weakly Lindel\"of spaces of arbitrarily large cardinality. 

Arhangel'ski\u\i\ \cite{A1} got closer to obtaining a common generalization of these two fundamental results by introducing a relative version of the weak Lindel\"of degree, namely  the cardinal invariant $wL_c(X)$, i.e.\ the least cardinal $\kappa $ such that for any closed set $F$ and any family of open sets $\mathcal{U}$  satisfying $F \subseteq \bigcup \mathcal{U}$ there is a subcollection $\mathcal{V} \in [\mathcal  {U}]^{\le \kappa }$ such that $F \subseteq \overline{\bigcup \mathcal{V}}$.

\begin{theorem} \cite{A1} If $X$ is a regular space, then $|X|\le 2^{wL_c(X)\chi(X)}$. 
\end{theorem}

O. Alas \cite{Al} showed that the previous inequality continues to hold for Urysohn spaces, but it is still open whether it's true for Hausdorff spaces.

In \cite{Ag} Arhangel'skii made another step ahead by introducing the notion of strict quasi-Lindel\"of degree, which allowed him to give a common refinement of \emph{the countable case} of his 1969 theorem and the Hajnal-Juh\'asz inequality. He defined a space $X$ to be \emph{strict quasi-Lindel\"of} if for every closed subset $F$ of $X$, for every open cover $\mathcal{U}$ of $F$ and for every countable decomposition $\{\mathcal{U}_n: n < \omega \}$ of $\mathcal{U}$ there are countable subfamilies $\mathcal{V}_n \subset \mathcal{U}_n$, for every $n< \omega$ such that $F \subset \bigcup \{\overline{\bigcup \mathcal{V}_n}: n < \omega \}$. It is easy to see that every Lindel\"of space is strict quasi-Lindel\"of and every ccc space is strict-quasi Lindel\"of. Arhangel'skii proved that every strict quasi-Lindel\"of first-countable space has cardinality at most continuum.

However, Arhangel'skii's approach cannot be extended to higher cardinals. Indeed, it's not even clear whether $|X| \leq 2^{\chi(X)}$ is true for every strict quasi-Lindel\"of space $X$. This inspired us to introduce the following cardinal invariants:

\begin{definition} {\ \\}
\begin{itemize}
\item The \emph{piecewise weak Lindel\"of degree of $X$} ($pwL(X)$) is defined as the minimum cardinal $\kappa$ such that for every open cover $\mathcal{U}$ of $X$ and every decomposition $\{\mathcal{U}_i: i \in I \}$ of $\mathcal{U}$, there are $\leq \kappa$-sized families $\mathcal{V}_i \subset \mathcal{U}_i$, for every $i \in I$ such that $X \subset \bigcup \{\overline{\bigcup \mathcal{V}_i}: i \in I \}$. 
\item The \emph{piecewise weak Lindel\"of degree for closed sets of $X$} ($pwL_c(X)$)  is defined as the minimum cardinal $\kappa$ such that for every closed set $F \subset X$, for every open family $\mathcal{U}$ covering $F$ and for every decomposition $\{\mathcal{U}_i: i \in I \}$ of $\mathcal{U}$, there are $\leq \kappa$-sized subfamilies $\mathcal{V}_i \subset \mathcal{U}_i$ such that $F \subset \bigcup \{\overline{\bigcup \mathcal{V}_i}: i \in I \}$.
\end{itemize}
\end{definition}

As a corollary to our main result, we will obtain the following bound, which is the desired common extension of Arhangel'skii's Theorem and the Hajnal-Juh\'asz inequality.

\begin{theorem}
For every Hausdorff space $X$, $|X| \leq 2^{pwL_c(X) \cdot \chi(X)}$.
\end{theorem}

For undefined notions we refer to \cite{En}. Our notation regarding cardinal functions mostly follows \cite{HChapter}. To state our proofs in the most elegant and compact way we use the language of elementary submodels, which is well presented in \cite{Dow}.

\section{A cardinal bound for the $G_\delta$-modification}

The following proposition collects a few simple general facts about the piecewise weak Lindel\"of number which will be helpful in the proof of the main theorem.

\begin{proposition} \label{prop}

For any space $X$ we have:

\begin{enumerate}
\item $pwL(X) \leq pwL_c(X)$.
\item $pwL_c(X)\leq L(X)$. 
\item $pwL_c(X)\le c(X)$.
\item \label{reg} If $X$ is $T_3$ then $wL_c(X) \leq pwL(X)$.
\end{enumerate}
\end{proposition}

\begin{proof} 

The first two items are trivial. To prove the third item, let $F$ be a closed subset of  $X$ and $\mathcal{V}=\bigcup\{\mathcal{V}_i: i \in I \}$ be an open collection
satisfying $F \subseteq \bigcup \mathcal{V}$. Suppose $c(X) \leq \kappa $.  For every $i \in I$  let $\mathcal{C}_i$ be a maximal collection of pairwise disjoint non-empty open subsets of $X$ such that  for each $C \in \mathcal{C}_i$ there is some $V_C \in \mathcal{V}_i $ with $C \subseteq V_C$.  By letting $\mathcal{W}_i =\{V_C:C \in \mathcal{C}_i \}$, the maximality of  $\mathcal{C}_i$ implies that $\bigcup \mathcal{V}_i \subseteq \overline{\bigcup \mathcal{W}_i }$ and so $F \subseteq \bigcup\{\overline {\cup \mathcal{W}_i }: i \in I \}$. Since $|\mathcal{W}_i | \leq |\mathcal{C}_i | \leq \kappa $, we have $pwL_c(X) \leq \kappa$. 

To prove the fourth item assume $X$ is a regular space and let $\kappa$ be a cardinal such that $pwL(X) \leq \kappa$. Let $F$ be a closed subset of $X$ and $\mathcal{U}$ be an open cover of $F$. If $\mathcal{U}$ covers $X$ we're done. Otherwise use regularity to choose, for every $p \in X \setminus \bigcup \mathcal{U}$ an open set $U_p$ such that $p \in U_p$ and $F \cap \overline{U}_p=\emptyset$. Note that $\mathcal{U} \cup \{U_p: p \in X \setminus F \}$ is an open cover of $X$, so by $pwL(X) \leq \kappa$, there is a $\kappa$-sized subfamily $\mathcal{V}$ of $\mathcal{U}$ such that $X  \subset  \overline{\bigcup \mathcal{V}} \cup \bigcup \{\overline{U_p}: p \in X \setminus F \}$. Hence $F \subset \overline{\bigcup \mathcal{V}}$ and we are done.
\end{proof}

\begin{corollary} \label{regcor}
If $X$ is a regular space then $|X| \leq 2^{pwL(X) \cdot \chi(X)}$.
\end{corollary}

\begin{proof}
Combine Proposition $\ref{prop}$, ($\ref{reg}$) and Arhangel'skii's result that $|X| \leq 2^{wL_c(X) \cdot \chi(X)}$ for every regular space $X$.
\end{proof}

We state our main theorem in terms of the \emph{$G_\kappa$-modification} of a space. Let $\kappa$ be a cardinal number. By $X_\kappa$ we denote the topology on $X$ generated by $\kappa$-sized intersections of open sets of $X$. We call $X_\kappa$, the \emph{$G_\kappa$-modification of $X$}; in case $\kappa=\omega$ we speak of the $G_\delta$-modification of $X$ and we often use the symbol $X_\delta$ instead. This construction has been extensively studied in the literature; various authors have tried to bound the cardinal functions of $X_\kappa$ in terms of their values on $X$ (see, for example \cite{BS}, \cite{FW}, \cite{J}, \cite{S}, \cite{SSz}) and results of this kind have found applications to other topics in topology, like the estimation of the cardinality of compact homogeneous spaces (see \cite{Agd}, \cite{BS}, \cite{CPR} and \cite{SSz}).

By $X^c_\kappa$ we denote the topology on $X$ generated by $G^c_\kappa$-sets, that is those subsets $G$ of $X$ such that there is a family $\{U_\alpha: \alpha < \kappa \}$ of open sets with $G=\bigcap \{U_\alpha: \alpha < \kappa\}=\bigcap \{\overline{U_\alpha}: \alpha < \kappa\}$. In general, the topology of $X^c_\kappa$ is coarser than the $G_\kappa$-modification of $X$, but if $X$ is a regular space then $X^c_\kappa=X_\kappa$.

\begin{theorem} \label{mainthm}
Let $X$ be a Hausdorff space such that $t(X) \cdot pwL_c(X) \leq \kappa$ and $X$ has a dense set of points of character $\leq \kappa$. Then $wL(X^c_\kappa) \leq 2^\kappa$.
\end{theorem}

\begin{proof}
Let $\mathcal{F}$ be a cover of $X$ by $G^c_\kappa$-sets. Let $\theta$ be a large enough regular cardinal and $M$ be a $\kappa$-closed elementary submodel of $H(\theta)$ such that $|M|=2^\kappa$ and $M$ contains everything we need (that is, $X, \mathcal{F} \in M$, $\kappa+1 \subset M$ etc...).

For every $F \in \mathcal{F}$ choose open sets $\{U_\alpha: \alpha < \kappa \}$ such that $F=\bigcap \{U_\alpha: \alpha < \kappa \}=\bigcap \{\overline{U_\alpha}: \alpha < \kappa \}$.

\medskip

\noindent {\bf Claim 1.} $\mathcal{F} \cap M$ covers $\overline{X \cap M}$.

\begin{proof}[Proof of Claim 1] Let $x \in \overline{X \cap M}$. Since $\mathcal{F}$ is a cover of $X$ we can find a set $F \in \mathcal{F}$ such that $x \in F$. Moreover, using $t(X) \leq \kappa$, we can find a $\kappa$-sized subset $S$ of $X \cap M$ such that $x \in \overline{S}$. Note that $x \in \overline{U_\alpha \cap S}$, for every $\alpha < \kappa$. Moreover, by $\kappa$-closedness of $M$, the set $U_\alpha \cap S$ belongs to $M$. Set $B=\bigcap \{\overline{U_\alpha \cap S}: \alpha < \kappa\}$. Note that $x \in B \subset F$ and $B \in M$. Therefore $ H(\theta) \models (\exists G \in \mathcal{F})(x \in B \subset G)$ and all the free variables in the previous formula belong to $M$. Therefore, by elementarity we also have that $M \models (\exists G \in \mathcal{F})(x \in B \subset G)$ and hence there exists a set $G \in \mathcal{F} \cap M$ such that $x \in G$, which is what we wanted to prove.

\renewcommand{\qedsymbol}{$\triangle$}
\end{proof}

\noindent {\bf Claim 2.} $\mathcal{F} \cap M$ has dense union in $X$.

\begin{proof}[Proof of Claim 2] Suppose by contradiction that $X \nsubseteq \overline{\bigcup (\mathcal{F} \cap M)}$. Then we can fix a point $p \in X \setminus \overline{\bigcup (\mathcal{F} \cap M)}$ such that $\chi(p,X) \leq \kappa$. Let $\{V_\alpha: \alpha < \kappa \}$ be a local base at $p$.

For every $F \in \mathcal{F} \cap M$, let $\{U_\alpha(F): \alpha < \kappa \} \in M$ be a sequence of open sets such that $F=\bigcap \{U_\alpha(F): \alpha < \kappa \}=\bigcap \{\overline{U_\alpha(F)}: \alpha < \kappa \}$. Note that $\{U_\alpha(F): \alpha < \kappa \} \subset M$. Let $\mathcal{C}=\{U_\alpha(F): F \in \mathcal{F} \cap M, \alpha < \kappa \}$. Note that $\mathcal{C}$ is an open cover of $\overline{X \cap M}$ and $\mathcal{C} \subset M$.

For every $x \in \overline{X \cap M}$, we can choose, using Claim 1, a set $F_x \in \mathcal{F} \cap M$ such that $x \in F_x$. Since $p \notin F_x$, there is $\alpha < \kappa$ such that $p \notin \overline{U_\alpha(F_x)}$. Hence we can find an ordinal $\beta_x < \kappa$ such that $V_{\beta_x} \cap U_\alpha(F_x)=\emptyset$. This shows that $\mathcal{U}=\{U \in \mathcal{C}: (\exists \beta < \kappa)(U \cap V_\beta=\emptyset) \}$ is an open cover of $\overline{X \cap M}$. Let $\mathcal{U}_\alpha=\{U \in \mathcal{U}: U \cap V_\alpha=\emptyset\}$. Then $\{\mathcal{U}_\alpha: \alpha < \kappa \}$ is a decomposition of $\mathcal{U}$ and hence we can find a $\kappa$-sized family $\mathcal{V}_\alpha \subset \mathcal{U}_\alpha$ for every $\alpha < \kappa$ such that $\overline{X \cap M} \subset \bigcup \{\overline{\bigcup \mathcal{V}_\alpha}: \alpha < \kappa \}$. Note that by $\kappa$-closedness of $M$ the sequence $\{\overline{\bigcup \mathcal{V}_\alpha}: \alpha < \kappa \}$ belongs to $M$ and hence the previous formula implies that:

$$M \models X \subset \bigcup \{\overline{\bigcup \mathcal{V}_\alpha}: \alpha < \kappa \}$$

So, by elementarity:

$$H(\theta) \models X \subset \bigcup \{\overline{\bigcup \mathcal{V}_\alpha}: \alpha < \kappa \}$$

But that is a contradiction, because $p \notin \overline{\bigcup \mathcal{V}_\alpha}$, for every $\alpha < \kappa$.

\renewcommand{\qedsymbol}{$\triangle$}
\end{proof}

Since $|\mathcal{F} \cap M| \leq 2^\kappa$, Claim 2 proves that $wL(X^c_\kappa) \leq 2^\kappa$, as we wanted.

\end{proof}

As a first consequence, we derive the desired common extension of Arhangel'skii's Theorem and the Hajnal-Juh\'asz inequality.

Recall that the \emph{closed pseudocharacter of the point $x$ in $X$} ($\psi_c(x,X)$) is defined as the minimum cardinal $\kappa$ such that there is a $\kappa$-sized family $\{ U_\alpha: \alpha < \kappa\}$ of open neighbourhoods of $x$ with $\bigcap \{\overline{U_\alpha}: \alpha < \kappa \}=\{x\}$. The closed pseudocharacter of $X$ ($\psi_c(X)$) is then defined as $\psi_c(X)= \sup \{\psi_c(x,X): x \in X \}$.

\begin{corollary} \label{maincor}
Let $X$ be a Hausdorff space. Then $|X| \leq 2^{pwL_c(X) \cdot \chi(X)}$.
\end{corollary}

\begin{proof}
It suffices to note that in a Hausdorff space $\psi_c(X) \cdot t(X) \leq \chi(X)$ and hence if $\kappa$ is a cardinal such that $\chi(X) \leq \kappa$ then $X^c_\kappa$ is a discrete set. Thus $wL(X^c_\kappa) \leq 2^\kappa$ if and only if $|X|=|X^c_\kappa| \leq 2^\kappa$.
\end{proof}

\noindent {\bf Remark}. Corollary $\ref{maincor}$ is a \emph{strict} improvement of both Arhangel'skii's Theorem and the Hajnal-Juh\'asz inequality. Indeed, if $S$ is the Sorgenfrey line and $A([0,1])$ the Aleksandroff duplicate of the unit interval, then  the space $X=(S \times S) \oplus A([0,1])$ is first countable, $pwL_c(X)=\aleph_0$ and $L(X)=c(X)=\mathfrak c$.

\medskip

Recall that a space is initially $\kappa$-compact if every open cover of cardinality $\leq \kappa$ has a finite subcover (for $\kappa=\omega$ we obtain the usual notion of countable compactness). The following Lemma essentially says that if $X$ is an initially $\kappa$-compact spaces such that $wL_c(X) \leq \kappa$, then it satisfies the definition of $pwL_c(X) \leq \kappa$ when restricted to decompositions of cardinality at most $\kappa$.

\begin{lemma} \label{init}
Let $X$ be an initially $\kappa$-compact space such that $wL_c(X) \leq \kappa$ and $F$ be a closed subset of $X$. If $\mathcal{U}$ is an open cover of $F$ and $\{\mathcal{U}_\alpha: \alpha < \kappa \}$ is a $\kappa$-sized decomposition of $\mathcal{U}$, then there are $\kappa$-sized subfamilies $\mathcal{V}_\alpha \subset \mathcal{U}_\alpha$ such that $F \subset \bigcup \{\overline{\bigcup \mathcal{V}_\alpha}: \alpha < \kappa \}$
\end{lemma}

\begin{proof}
Let $U_\alpha=\bigcup \mathcal{U}_\alpha$. Then $\{U_\alpha: \alpha < \kappa \}$ is an open cover of $F$ of cardinality $\kappa$, so by initial $\kappa$-compactness there is a finite subset $S$ of $\kappa$ such that $F \subset \{U_\alpha: \alpha \in S \}$. Let now $\mathcal{W}=\bigcup \{\mathcal{U}_\alpha: \alpha  \in S \}$. We then have $F \subset \bigcup \mathcal{W}$ and hence by $wL_c(X) \leq \kappa$ we can find a $\kappa$-sized subfamily $\mathcal{W}'$ of $\mathcal{W}$ such that $F \subset \overline{\bigcup \mathcal{W}'}$. Set now $\mathcal{V}_\alpha=\{W \in \mathcal{W}': W \in \mathcal{U}_\alpha \}$. Then $|\mathcal{V}_\alpha| \leq \kappa$ and $F \subset \bigcup \{\overline{\bigcup \mathcal{V}_\alpha}: \alpha  <\kappa\}$, as we wanted.
\end{proof}

Noticing that in the proof of Theorem $\ref{mainthm}$ we only needed to apply the definition of $pwL_c(X) \leq \kappa$ to decompositions of cardinality $\kappa$, Theorem $\ref{mainthm}$ and Lemma $\ref{init}$ imply the following corollaries.

\begin{corollary}
\cite{BS} Let $X$ be an initially $\kappa$-compact space containing a dense set of points of character $\leq \kappa$ and such that $wL_c(X) \cdot t(X) \leq \kappa$. Then $wL(X^c_\kappa) \leq 2^\kappa$.
\end{corollary}

\begin{corollary}
(Alas, \cite{Al}) Let $X$ be an initially $\kappa$-compact space with a dense set of points of character $\kappa$, such that $wL_c(X) \cdot t(X) \cdot \psi_c(X) \leq \kappa$. Then $|X| \leq 2^\kappa$.
\end{corollary}

\section{Open Questions}

Corollary $\ref{regcor}$ can be slightly improved by replacing regularity with the Urysohn separation property (that is, every pair of distinct points can be separated by disjoint closed neighbourhoods). Indeed, in a similar way as in the proof of Proposition $\ref{prop}$ ($\ref{reg}$) it can be shown that if $X$ is Urysohn then $wL_\theta(X) \leq pwL(X)$, where $wL_\theta(X)$ is the weak Lindel\"of number for \emph{$\theta$-closed sets} (see \cite{BC}). Moreover, $|X| \leq 2^{wL_\theta(X) \cdot \chi(X)}$ for every Urysohn space $X$. However it's not clear whether regularity can be weakened to the Hausdorff separation property. That motivates the next question.

\begin{question}
Is the inequality $|X| \leq 2^{pwL(X) \cdot \chi(X)}$ true for every Hausdorff space $X$? 
\end{question}

Moreover, we were not able to find an example which distinguishes countable piecewise weak Lindel\"of number for closed sets from the strict quasi-Lindel\"of property.

\begin{question}
Is there a strict quasi-Lindel\"of space $X$ such that $pwL_c(X) > \aleph_0$?
\end{question}

Finally, Arhangel'skii's notion of a strict quasi-Lindel\"of space suggests a natural cardinal invariant. Define the strict quasi-Lindel\"of number of $X$ ($sqL(X)$) to be the least cardinal number $\kappa$, such that for every closed subset $F$ of $X$, for every open cover $\mathcal{U}$ of $F$ and for every \emph{$\kappa$-sized} decomposition $\{\mathcal{U}_\alpha: \alpha < \kappa \}$ of $\mathcal{U}$ there are $\kappa$-sized subfamilies $\mathcal{V}_\alpha \subset \mathcal{U}_\alpha$ such that $X \subset \bigcup \{\overline{\bigcup \mathcal{V}_\alpha}: \alpha < \kappa \}$. Obviously $sqL(X) \leq pwL_c(X)$. It's not at all clear from our argument whether the piecewise weak-Lindel\"of number for closed sets can be replaced with the strict quasi-Lindel\"of number in Corollary $\ref{maincor}$.

\begin{question}
Let $X$ be a Hausdorff space. Is it true that $|X| \leq 2^{sqL(X) \cdot \chi(X)}$?
\end{question}

Even the following special case of the above question seems to be open.

\begin{question}
Let $X$ be a strict quasi-Lindel\"of space. Is it true that $|X| \leq 2^{\chi(X)}$?
\end{question}

\end{document}